\documentclass{article}

\usepackage{etex}  
\usepackage{xcolor}
\usepackage{amsthm}
\usepackage{amsfonts}
\usepackage{amssymb}
\usepackage{amsgen}
\usepackage{amsmath}
\usepackage{amsopn}
\usepackage{verbatim}
\usepackage{xypic}
\usepackage{pgf}
\usepackage{xspace}
\usepackage{multicol}
\usepackage{makeidx}
\usepackage{eepic}
\usepackage{upref}
\usepackage{pgf}
\usepackage{tikz}
\usepackage[normalem]{ulem}
\usepackage{shuffle,yfonts}
\DeclareFontFamily{U}{shuffle}{}
\DeclareFontShape{U}{shuffle}{m}{n}{ <-8>shuffle7 <8->shuffle10}{}

\allowdisplaybreaks

\newcommand{\FES}{\mathsf {FES}}

\newcommand{\FMTV}{\mathsf {FMT}}

\newcommand{\FAMMV}{\mathsf {FAM}}
\newcommand{\FAMTV}{\mathsf {FAMT}}
\newcommand{\FCMZV}{\mathsf {FCMZ}}
\newcommand{\ES}{\mathsf {ES}}

\newcommand{\MTV}{\mathsf {MT}}

\newcommand{\AMTV}{\mathsf {AMT}}
\newcommand{\sha}{\shuffle}

\newcommand{\Sy}{{\mathcal S}}

\newcommand{\gb}{\beta}
\newcommand{\gd}{\delta}
\newcommand{\ola}{\overleftarrow}
\newcommand\fA{{\mathfrak{A}}}
\newcommand\fT{{\mathfrak{T}}}

\newcommand\ta{{\texttt{a}}}
\newcommand\tb{{\texttt{b}}}
\newcommand\tc{{\texttt{c}}}
\newcommand{\tq}{{\texttt{q}}}
\newcommand\tx{{\texttt{x}}}
\newcommand\ty{{\texttt{y}}}
\newcommand\emtx{{\emph{\texttt{x}}}}
\newcommand\emty{{\emph{\texttt{y}}}}
\newcommand{\emtq}{{\emph{\texttt{q}}}}

\newcommand\gs{{\sigma}}
\newcommand\eps{{\varepsilon}}
\newcommand{\bfp}{{\bf p}}
\newcommand{\bfq}{{\bf q}}
\newcommand{\bfu}{{\bf u}}
\newcommand{\bfv}{{\bf v}}
\newcommand{\bfw}{{\bf w}}
\newcommand{\bfs}{{\boldsymbol{\sl{s}}}}
\newcommand{\bft}{{\boldsymbol{\sl{t}}}}
\newcommand\bfeps{{\boldsymbol \varepsilon}}

\newcommand\bfgs{{\boldsymbol \sigma}}

\newcommand{\calA}{\mathcal{A}}
\newcommand{\calC}{\mathcal{C}}
\newcommand{\calP}{\mathcal{P}}

\allowdisplaybreaks
\DeclareMathOperator*{\sgn}{sgn}
\DeclareMathOperator*{\dep}{dep}

\newcommand{\N}{\mathbb{N}}
\newcommand{\Z}{\mathbb{Z}}
\newcommand{\Q}{\mathbb{Q}}

\newcommand{\ol}{\overline}
\newcommand{\db}{\mathbb{D}}
\def\ppmod#1{{\ (\rm{mod}\ 2)}}

\theoremstyle{plain}
\newtheorem{thm}{Theorem}[section]

\newtheorem{conj}[thm]{Conjecture}
\newtheorem{cpp}[thm]{Conjecture-Principle-Philosophy}
\newtheorem{cor}[thm]{Corollary}
\newtheorem{prop}[thm]{Proposition}

\theoremstyle{definition}
\newtheorem{defn}{Definition}[section]
\newtheorem{rem}[thm]{Remark}
\newtheorem{eg}[thm]{Example}
\newtheorem{prob}[thm]{Problem}

\begin{document}
\title{Finite and Symmetric Euler Sums and Finite and Symmetric (Alternating) Multiple $T$-Values}

\author{Jianqiang Zhao\footnote{Email:zhaoj@ihes.fr}\\
\ \\
Department of Mathematics, The Bishop's School, La Jolla, CA 92037, USA} 

\date{}

\maketitle

\medskip
\noindent
\textbf{Abstract.}
In this paper, we will study finite multiple $T$-values (MTVs) and their alternating versions, which are level two and level four variations of finite multiple zeta values, respectively. We will first provide some structural results for level two finite multiple zeta values (i.e., finite Euler sums) for small weights, guided by the author's previous conjecture that the finite Euler sum space of weight, $w$, is isomorphic to a quotient Euler sum space of weight, $w$. Then, by utilizing some well-known properties of the classical alternating MTVs, we will derive a few important $\Q$-linear relations among the finite alternating MTVs, including the reversal, linear shuffle, and sum relations. We then compute the upper bound for the dimension of the $\Q$-span of finite (alternating) MTVs for some small weights by rigorously using the newly discovered relations, numerically aided by computers.
 


\section{Introduction}

In~\cite{KanekoTs2020}, Kaneko and Tsumura proposed a study of  \emph{multiple $T$-values} (MTVs):
\begin{align}\label{defn:MTV}
T(\bfs):= \sum_{\substack{n_1>\dots>n_d>0\\ n_j\equiv d-j+1\ppmod{2}}} \prod_{j=1}^d \frac{1}{n_j^{s_j}}, \quad \bfs=(s_1,\dots,s_d)\in\N^d,
\end{align}
as level two variations of \emph{multiple zeta values}
(MZVs), which, in turn, were first studied
by Zagier~\cite{Zagier1994} and Hoffman~\cite{Hoffman1992} independently:
\begin{equation}\label{defn:MZV}
\zeta(\bfs):=\sum_{n_1>\dots>n_d>0} \prod_{j=1}^d \frac{1}{n_j^{s_j}}, \quad \bfs=(s_1,\dots,s_d)\in\N^d,
\end{equation}
where $\N$ is the set of positive integers.
These series converge if and only if $s_1\ge 2$, in which case we say $\bfs$ is admissible. As~usual, we call
$|\bfs|:=s_1+\cdots+s_d$ the weight and $d$ the depth. Note that the series becomes Riemann zeta values
when $d=1$. One of the most important properties of these values is that they can be expressed by iterated integrals.
The main motivation to consider MTVs is that they
are also equipped with the following iterated integral expressions:\vspace{-6pt}
\begin{equation}\label{equ:MTVitIntegral}
T(\bfs)=\int_0^1  \left(\frac{dt}{t}\right)^{s_1-1}\frac{dt}{1-t^2} \cdots \left(\frac{dt}{t}\right)^{s_d-1}\frac{dt}{1-t^2}
\end{equation}
which provide the MTVs with a $\Q$-algebra structure because of the shuffle product property satisfied by
iterated integral multiplication (see, e.g.,~\cite[Lemma~2.1.2(iv)]{ZhaoBook}).

In addition to MTVs, many other variants of MZVs have been studied due to their
important connections to a variety of objects in both mathematics and theoretical physics.
For example, Yamamoto~\cite{Yamamoto2012b} defined the interpolated
version of MZVs, which connects ordinary MZVs to the starred version;
Hoffman~\cite{Hoffman2019} defined an odd variant by restricting the summation
indices $n_j$'s to odd numbers only; Xu and the author~\cite{XuZhao2020a,XuZhao2020c} further extended
both Kaneko-Tsumura and Hoffman's versions to allow for all possible parity~patterns.

On the other hand, the~congruence properties of the partial sums of MZVs were first considered by
Hoffman~\cite{Hoffman2015} and the author \cite{Zhao2008a} independently.
Contrary to the classical cases,
only a few variants of these sums exist (see, e.g.,~\cite{KanekoMuYo2021,Murahara2014,SingerZ2019}).
In this paper, the author will concentrate on the finite analog of MTVs defined by \eqref{defn:MTV}.

Let $\calP$ be the set of primes. Then by putting\vspace{-6pt}
\begin{equation}\label{equ:poormanA}
\calA:=\prod_{p\in\calP}(\Z/p\Z)\bigg/\bigoplus_{p\in\calP}(\Z/p\Z),
\end{equation}
we can define the \emph{finite multiple zeta values} (FMZVs) according to the following:
\begin{equation*}
\zeta_\calA(\bfs):=\bigg(\sum_{p>n_1>\dots>n_d>0}
 \prod_{j=1}^d \frac{1}{n_j^{s_j}} \pmod{p}\bigg)_{p\in\calP}\in \calA.
\end{equation*}

Nowadays, the~main motivation for studying FMZVs is to understand a deep conjecture
proposed by Kaneko and Zagier around 2014 (see Conjecture~\ref{conj:KanekoZagierAltVersion} below for a generalization).
Although this conjecture is far from being proved, many parallel results have been shown to hold for both
MZVs and FMZVs simultaneously (see, e.g.,~\cite{MuraharaSa2019,SaitoWa2013b,Sakurada2023}).
In particular, for~each positive integer $w\ge 2$, the~element
\begin{equation}\label{equ:betaw}
   \gb_w:=\Big(\frac{B_{p-w}}{w}\Big)_{w<p\in\calP} \in \calA
\end{equation}
is the finite analog of $\zeta(w)$, where $B_n$'s are the Bernoulli numbers defined by
\begin{equation*}
    \frac{t}{e^t-1}=\sum_{n\ge 0}B_n\frac{ t^n}{n!},
\end{equation*}
which have played very important roles in many areas of mathematical studies, such as Clifford analysis~\cite{ChandragiriSh2018} and topology~\cite{KervaireMilnor1958}.

Furthermore, the connection goes even further to their alternating versions --- the Euler sums and finite Euler sums.
For $s_1,\dots,s_d\in\N$ and $\gs_1,\dots,\gs_d=\pm 1$, we define the \emph{Euler sums}  % This is necessary.
\begin{equation*}
\zeta(s_1,\dots,s_d;\gs_1,\dots,\gs_d)
   := \sum_{n_1>\cdots>n_d>0}\;\prod_{j=1}^d  \frac{\gs_j^{n_j}}{n_j^{s_j}}.
\end{equation*}

In order to save space, if~$\gs_j=-1$, then $\overline s_j$ will be used, and if a substring, $S$, repeats $n$
times in the list, then $\{S\}^n$ will be used.
For example, the~finite analog of $-\zeta(\bar1)=-\zeta(1;-1)=\log 2$ is the Fermat quotient
\begin{equation}\label{equ:FermatQ2}
  \tq_2:=\Big(\frac{2^{p-1}-1}{p} \pmod{p}\Big)_{3\le p\in\calP} \in \calA.
\end{equation}
Write $\sgn(\bar s)=-1$ and $|\bar s|=s$ if $s\in\N$. For~$s_1,\ldots, s_d\in\db:=\N\cup\bar{\N}$,
we can define the finite \emph{Euler sums} as
\begin{equation*}
\zeta_\calA(\bfs):=\bigg(\sum_{p>n_1>\dots>n_d>0}
 \prod_{j=1}^d \frac{\sgn(s_j)^{n_j}}{n_j^{|s_j|}} \pmod{p}\bigg)_{p\in\calP}\in \calA.
\end{equation*}

In \cite[Conjecture~8.6.9]{ZhaoBook}, we extended the Kakeko--Zagier conjecture to the setting of the Euler sums.
For $\bfs=(s_1,\ldots, s_d)\in\db^d$, define the symmetric version of the alternating Euler sums\vspace{-6pt}
\begin{align*}
\zeta_\sharp^\Sy(\bfs):=&\sum_{i=0}^d
\bigg(\prod_{j=1}^i (-1)^{|s_j|} \sgn (s_j)\bigg) \zeta_\sharp(s_i,\dots,s_1) \zeta_\sharp(s_{i+1},\dots,s_d)
\end{align*}
where $\zeta_\sharp$ ($\sharp=\ast$ or $\sha$) are regularized values (see \cite[Proposition~13.3.8]{ZhaoBook}).
They are called $\sharp$-regularized \emph{symmetric Euler sums}. If~$\bfs\in\N^d$, then they are called
$\sharp$-regularized \emph{symmetric multiple zeta values} (SMZVs).

\begin{conj}[{cf. \cite[Conjecture~8.6.9]{ZhaoBook}}]\label{conj:KanekoZagierAltVersion}
For any $w\in\N$, let $\FES_{w}$ (resp.\ $\ES_w$) be the $\Q$-vector space generated by
all finite Euler sums (resp.\ Euler sums) of weight $w$. Then, there is an isomorphism:\vspace{-6pt}
\begin{align*}
f_\ES: \FES_{w} & \longrightarrow \frac{\ES_w}{\zeta(2)\ES_{w-2}} , \\
 \zeta_\calA(\bfs) & \longmapsto \zeta_\sharp^\Sy(\bfs),
\end{align*}
where $\sharp=\ast$ or $\sha$.
\end{conj}

We remark that $\zeta_\sha^\Sy(\bfs)-\zeta_\ast^\Sy(\bfs)$ always lies in $\zeta(2)\ES_{w-2}$, see \cite{ZhaoBook}, Exercise 8.7.
Thus, it does not matter which version of the symmetric Euler sums is used in the~conjecture.

\begin{prob}
What is the correct generalization of \cite[Theorem~6.3.5]{ZhaoBook} for symmetric Euler sums?
What is the correct extension of \cite[Theorem~8.5.10]{ZhaoBook} to finite Euler sums?
\end{prob}

Our primary motivation for studying alternating MTVs is to better understand this mysterious relation, $f_\ES$.
One of the main results of this paper is the discovery of the linear shuffle relations among the finite alternating
MTVs given by Theorem~\ref{thm:FMTVshuffle}. For~example, it immediately implies the highly nontrivial result in Proposition~\ref{prop:FMTV2d=0}: for all $d\in\N$, we have
\begin{equation*}
  T_\calA(\{1\}^{2d})=0.
\end{equation*}

We now briefly describe the content of this paper. We will start the next section by defining finite MTVs and
symmetric MTVs, which can be shown to appear on the two sides of Conjecture~\ref{conj:KanekoZagierAltVersion}, respectively.
The most useful property of MTVs is that they have the iterated integral expressions \eqref{equ:MTVitIntegral},
satisfying shuffle multiplication. This leads us to the discovery of the linear shuffle relations
for the finite MTVs (and their alternating version) in Section~\ref{sec:linShuffle} and some interesting applications of these relations. Section~\ref{s4} is devoted to presenting a few results about the alternating MTVs and providing their structures explicitly when the weight is one or two. In~the last section, we consider both the finite MTVs and their alternating version by computing the dimension
of the weight $w$ piece for $w<9$ and then compare these data to their Archimedean counterparts, as obtained
by Xu and the author~\cite{XuZhao2020a,XuZhao2020c}.

\section{Symmetric and Finite  Multiple $T$-Values}
It turns out that finite MTVs are closely related to another variant called finite MSVs.
For all admissible $\bfs=(s_1,\dots,s_d)\in\N^d$, we define the \emph{finite multiple $T$-values} (FMTVs)
and the \emph{finite multiple $S$-values} (FMSVs) as\vspace{-6pt}
\begin{align*}
F_\calA(\bfs):=\bigg(  \sum_{\substack{p>n_1>\dots>n_d>0\\ n_j\equiv d-j+1\ppmod{2} \ \text{if }F=T,\\ n_j\equiv d-j\ppmod{2} \ \text{if }F=S}} \prod_{j=1}^d \frac{1}{n_j^{s_j}} \pmod{p}\bigg)_{p\in\calP}\in\calA.
\end{align*}
It is clear that
\begin{equation*}
F_\calA(\bfs)=\frac{1}{2^d}\sum_{\gs_1,\dots,\gs_d=\pm 1}\bigg(\prod_{\substack{1\le j\le d\\ 2|d-j\ \text{if $F=T$}\\  2\nmid d-j\ \text{if $F=S$}}} \gs_j\bigg)  \zeta_\calA(\bfs;\bfgs).
\end{equation*}

Motivated by Conjecture~\ref{conj:KanekoZagierAltVersion}, we provide the following definition.
\begin{defn}
Let $d\in\N$ and $\bfs=(s_1,\dots,s_d)\in\N^d$. Let $F=S$ or $T$.
We define the \emph{$\sharp$-regularized MTVs}
($\sharp=\ast$ or $\sha$) and \emph{MSVs} as
\begin{align*}
F_\sharp(\bfs):=&\, \frac{1}{2^d}\sum_{\gs_1,\dots,\gs_d=\pm 1}\bigg(\prod_{\substack{1\le j\le d\\ 2|d-j\ \text{if $F=T$}\\  2\nmid d-j\ \text{if $F=S$}}} \gs_j\bigg)  \zeta_\sharp(\bfs;\bfgs) \quad (\text{$F=T$ or $S$}).
\end{align*}

We define the \emph{$\sharp$-symmetric multiple $T$-values} (SMTVs) and
 \emph{$\sharp$-symmetric multiple $S$-values} (SMSVs) as
\begin{equation*}
F_\sharp^\Sy(\bfs):=
\left\{
  \begin{array}{ll}
  \displaystyle \sum_{i=0}^d  \Big(\prod_{\ell=1}^i (-1)^{s_\ell} \Big) F_\sharp(s_i,\dots,s_1)F_\sharp(s_{i+1},\dots,s_d),\quad & \quad \hbox{if $d$ is even;} \\
  \displaystyle \sum_{i=0}^d  \Big(\prod_{\ell=1}^i (-1)^{s_\ell} \Big) \widetilde{F}_\sharp(s_i,\dots,s_1)F_\sharp(s_{i+1},\dots,s_d), \quad&\quad \hbox{if $d$ is odd,}
  \end{array}
\right.
\end{equation*}
where $\widetilde{F}=S+T-F$ and we set, as usual, $\prod_{\ell=1}^0=1$.
\end{defn}

\begin{prop} \label{prop:fESofTS}
Suppose $f_\ES$ is defined as per Conjecture \ref{conj:KanekoZagierAltVersion}.
Let $\sharp=\ast$ or $\sha$. Then, for all $\bfs=(s_1,\dots,s_d)\in\N^d$, we have
$f_\ES T_\calA(\bfs)=T_\sharp^\Sy(\bfs)$ and $f_\ES S_\calA(\bfs)=S_\sharp^\Sy(\bfs)$ modulo $\zeta(2)$.
\end{prop}

\begin{proof}
Suppose $d$ is even and $\bfs\in\N^d$. Then, modulo $\zeta(2)$
\begin{align*}
f_\ES T_\calA(\bfs)=&\, \frac{1}{2^d}\sum_{\eps_1,\dots,\eps_d=\pm 1}\bigg(\prod_{\substack{1\le j\le d\\ j\equiv d\ppmod{2}}} \eps_j\bigg)  f_\ES \zeta_\calA\binom{\bfs}{\bfeps}\\
=&\, \frac{1}{2^d}\sum_{\eps_1,\dots,\eps_d=\pm 1}\bigg(\prod_{\substack{1\le j\le d\\ 2|j}} \eps_j \bigg)\zeta_\sharp^\Sy\binom{\bfs}{\bfeps}\\
=&\, \frac{1}{2^d}\sum_{\eps_1,\dots,\eps_d=\pm 1}\bigg(\prod_{\substack{1\le j\le d\\ 2|j}} \eps_j\bigg)
    \sum_{i=0}^d  \Big(\prod_{\ell=1}^i(-1)^{s_\ell} \eps_\ell\Big) \zeta_\sharp\binom{s_i,\dots,s_1}{\eps_i,\dots,\eps_1}\zeta_\sharp\binom{s_{i+1},\dots,s_d}{\eps_{i+1},\dots,\eps_d}\\
=&\, \frac{1}{2^d} \sum_{i=0}^d \bigg(\prod_{\ell=1}^i  (-1)^{s_\ell}\bigg)
\sum_{\eps_1,\dots,\eps_d=\pm 1}\bigg(\prod_{\substack{1\le j\le d\\ 2|j}} \eps_j\bigg)
  \bigg(\prod_{\ell=1}^i \eps_\ell\bigg) \zeta_\sharp\binom{s_i,\dots,s_1}{\eps_i,\dots,\eps_1}\zeta_\sharp\binom{s_{i+1},\dots,s_d}{\eps_{i+1},\dots,\eps_d}\\
=&\, \frac{1}{2^d} \sum_{i=0}^d  \Big(\prod_{\ell=1}^i (-1)^{s_\ell} \Big)\bigg(\sum_{\eps_1,\dots,\eps_i=\pm 1}
\prod_{\substack{1\le j\le i\\ 2\nmid j}} \eps_j
    \zeta_\sharp\binom{s_i,\dots,s_1}{\eps_i,\dots,\eps_1}\bigg) \\
&\, \hskip6cm \times \bigg(\sum_{\eps_1,\dots,\eps_i=\pm 1}
\prod_{\substack{i< j\le d\\ 2|j}} \eps_j \zeta_\sharp\binom{s_{i+1},\dots,s_d}{\eps_{i+1},\dots,\eps_d}\bigg)\\
=&\, \frac{1}{2^d} \sum_{i=0}^d  \Big(\prod_{\ell=1}^i (-1)^{s_\ell} \Big)T_\sharp(s_i,\dots,s_1)T_\sharp(s_{i+1},\dots,s_d)\\
=&\, T_\sharp^\Sy(\bfs).
\end{align*}

The MSVs and the odd $d$ cases can all be computed similarly and are left to the interested reader.
\end{proof}

Hence, we expect that whenever certain relations hold on the finite side, then the same relations should hold for the
symmetric version, at~least modulo $\zeta(2)$, and~vice~versa. Sometimes, they are valid for the symmetric version even
without modulo $\zeta(2)$. For~example, the~following reversal relations hold for both types of sums by \cite[Propositions~2.8 and 2.9]{FKLQWZ2024}. For~$\bfs=(s_1,\dots,s_d)$, we state $\ola{\bfs}=(s_d,\dots,s_1)$.

\begin{prop}[{Reversal relations}]\label{prop:FMMVreversal}
For all $\bfs\in\N^d$, if~$d$ is even, then
\begin{align*}
T_\calA(\ola{\bfs})=&\, (-1)^{|\bfs|} T_\calA(\bfs) \quad\text{and}\quad S_\calA(\ola{\bfs})=(-1)^{|\bfs|} S_\calA(\bfs), \\
T_\ast^\Sy(\ola{\bfs})=&\, (-1)^{|\bfs|} T_\ast^\Sy(\bfs) \quad\text{and}
\quad S_\ast^\Sy(\ola{\bfs})=(-1)^{|\bfs|} S_\ast^\Sy(\bfs),
\end{align*}
and if $d$ is odd, then
\begin{align*}
T_\calA(\ola{\bfs})=&\, (-1)^{|\bfs|} S_\calA(\bfs)\quad\text{and}\quad S_\calA(\ola{\bfs})=(-1)^{|\bfs|} T_\calA(\bfs) ,\\
T_\ast^\Sy(\ola{\bfs})=&\, (-1)^{|\bfs|} S_\ast^\Sy(\bfs)\quad\text{and}
\quad S_\ast^\Sy(\ola{\bfs})=(-1)^{|\bfs|} T_\ast^\Sy(\bfs).
\end{align*}
\end{prop}

\section{Linear Shuffle Relations for Finite Alternating Multiple $T$-Values}\label{sec:linShuffle}
One of the most important tools for studying MZVs and Euler sums is to consider the double shuffle relations that
are produced in two ways to express these sums: one as series (by definition) and the~other as iterated integrals.
This idea will play the key role in the following discovery of the linear shuffle relations for
finite multiple $T$-values (FMTVs) and their alternating~version.

The linear shuffle relations for Euler sums are given in \cite[Theorem~8.4.3]{ZhaoBook}. First, we extend MTVs and FMTVs
to their alternating version. For~all admissible $(\bfs,\bfgs)\in\N^d\times \{\pm 1\}^d$ (i.e., $(s_1,\gs_1)\ne(1,1)$),
we define the alternating multiple $T$-values as\vspace{-4pt}
\begin{equation*}
T(\bfs;\bfgs):=\sum_{\substack{n_1>\dots>n_d>0\\ n_j\equiv d-j+1\ppmod{2}}} \prod_{j=1}^d \frac{\gs_j^{(n_j-d+j-1)/2}}{n_j^{s_j}}.
\end{equation*}
This is basically the same definition we used in \cite{XuYanZhao2022Aug}, except
for a possible sign difference. If~we denote the version in loc. cit. as $T'(\bfs;\bfgs)$, then
\begin{equation}\label{equ:oldNewConvert}
T(\bfs;\bfgs)=T'(\bfs;\bfgs)\prod_{d-j \equiv 0,1 \ppmod{4}} \gs_j.
\end{equation}
We changed to our new convention in this paper because of the significant simplification in this special case.
However, the~old convention is still superior when treating the general alternating multiple mixed values.
Similar to the convention for Euler sums,
we will save space by putting a bar on top of $s_j$ if $\gs_j=-1$. For~example,
\begin{equation*}
T(\bar2,1)= \sum_{n>m>0}  \frac{(-1)^{n-1}}{(2n-2)^2(2m-1)}.
\end{equation*}

In order to study the alternating MTVs, it is to our advantage to consider the \emph{alternating multiple $T$-functions}
of one variable, as follows: for~any real number, $x$, define
\begin{equation*}
T(\bfs;\bfgs;x):=\sum_{\substack{n_1>\dots>n_d>0\\ n_j\equiv d-j+1\ppmod{2}}}
x^{n_1} \prod_{j=1}^d \frac{\gs_j^{(n_j-d+j-1)/2}}{n_j^{s_j}}.
\end{equation*}
In the non-alternating case, this function is the A-function (up to a power of 2) used by
Kaneko and Tsumura in~\cite{KanekoTs2020}.
For all $\eta_1,\dots,\eta_d=\pm 1$, it is then easy to evaluate the iterated integral:

\begin{align*}
&\, \int_0^x  \left(\frac{dt}{t}\right)^{s_1-1}\frac{ dt}{1-\eta_1 t^2} \cdots \left(\frac{dt}{t}\right)^{s_d-1}\frac{ dt}{1-\eta_d t^2} \\
=&\, \sum_{k_1>\dots>k_d>0} x^{2(k_1+\dots+k_d)+d} \prod_{j=1}^d \frac{\eta_j^{k_j}}{(2k_j+2k_{j+1}+\dots+2k_d+d-j+1)^{s_j}}.
\end{align*}

Let
\begin{equation*}
\ty_0=\frac{dt}{t}, \qquad \ty_1=\frac{dt}{1-t^2}, \qquad \ty_{-1}:=\frac{dt}{1+t^2}.
\end{equation*}
By changing the indices $n_j=2k_j+2k_{j+1}+\dots+2k_d+d-j+1$, we immediately obtain
\begin{equation*}
T(\bfs;\bfgs;x)=  \int_0^x \bfp\Big(\ty_0^{s_1-1} \ty_{\gs_1} \cdots \ty_0^{s_d-1} \ty_{\gs_d}\Big)
:= \int_0^x \ty_0^{s_1-1} \ty_{\eta_1} \cdots \ty_0^{s_d-1} \ty_{\eta_d},
\end{equation*}
where $\eta_j=\gs_1\cdots \gs_j$ for all $j\ge 1$ and  $\bfp,\bfq$
represent the conversions between the series and the integral expressions of
alternating MTVs:\vspace{-4pt}
\begin{align}\label{equ:bfp}
\bfp(\bfu):=&\, \ty_0^{s_1-1} \ty_{\gs_1}\dots\ty_0^{s_j-1} \ty_{\gs_1\cdots\gs_j}\dots\ty_0^{s_d-1} \ty_{\gs_1\cdots\gs_d},\\
\bfq(\bfu):=&\, \ty_0^{s_1-1} \ty_{\gs_1}\dots\ty_0^{s_j-1} \ty_{\gs_j/\gs_{j-1}}\dots\ty_0^{s_d-1} \ty_{\gs_d/\gs_{d-1}}. \label{equ:bfq}
\end{align}
Namely, $\bfp$ pushes a word used in the series definition to a word used in the integral expression,
whereas $\bfq$ goes backward. See \cite{XuYanZhao2022Aug} for more~details.

In order to state the linear shuffle relations among FMTVs and their alternating version,
first, we quickly review the algebra setup
and the corresponding results for Euler sums. Let $\fA_1^*$ (resp.\ $\fA_2^*$) be the
$\Q$-algebra of words on $\{\tx_0,\tx_1\}$ (resp.\ $\{\tx_0,\tx_1,\tx_{-1}\}$) with concatenation as the product.
Let $\fA_j^1$ ($j=1,2$) be the sub-algebra generated by the words not ending with $\tx_0$.
Then, for each word $\bfu=\tx_0^{s_1-1} \tx_{\eta_1}\dots\tx_0^{s_d-1} \tx_{\eta_d}\in \fA_2^1$, we define
\begin{equation*}
\zeta_\calA(\bfu):=\zeta_\calA (s_1,\dots,s_d;\gs_1,\dots,\gs_d)
\end{equation*}
where $\gs_1=\eta_1$ and $\gs_j=\eta_j/\eta_{j-1}$ for all $j\ge 2$. Set $\tau(1)=1$ and
\begin{equation*}
\tau(\tx_0^{s_1-1}\tx_1 \dots \tx_0^{s_d-1}\tx_1)=(-1)^{s_1+\dots+s_r} \tx_0^{s_d-1}\tx_1 \dots \tx_0^{s_1-1}\tx_1.
\end{equation*}

\begin{thm}[{\cite[Theorem~8.4.3]{ZhaoBook}}]\label{thm:FESshuffle} 
For all words, $\bfw,\bfu\in\fA_1^1$, $\bfv\in\fA_2^1$, and~ $s\in\N$, we~have
\begin{itemize}
 \item[{(i)}]
 $\zeta_\calA(\bfu \sha \bfv)=\zeta_\calA(\tau(\bfu)\bfv)$;

\item[{(ii)}]
 $\zeta_\calA( (\bfw\bfu) \sha \bfv)=\zeta_\calA(\bfu \sha\tau(\bfw)\bfv)$;

\item[{(iii)}]
 $\zeta_\calA( (\emtx_0^{s-1}\emtx_1 \bfu) \sha\bfv)
 =(-1)^s \zeta_\calA(\bfu \sha (\emtx_0^{s-1}\emtx_1\bfv)).$
\end{itemize}
\end{thm}

For alternating MTVs, we can similarly let $\fT_1^*$ (resp.\ $\fT_2^*$) be the
$\Q$-algebra of words on $\{\ty_0,\ty_1\}$ (resp.\ $\{\ty_0,\ty_1,\ty_{-1}\}$) with concatenation as the product.
Let $\fT_j^1$ ($j=1,2$) be the sub-algebra generated by the words not ending with $\ty_0$.
Then, for each word $\bfu=\ty_0^{s_1-1} \ty_{\gs_1}\dots\ty_0^{s_d-1} \ty_{\gs_d}\in \fA_2^1$,
let $\bfp,\bfq: \fA_2^1  \to \fA_2^1$ be the two maps defined in \eqref{equ:bfp} and \eqref{equ:bfq}.
Then, we can extend the definition of alternating MTVs and their corresponding
one-variable functions to the word level:\vspace{-6pt}
\begin{equation*}
F_\ast(\bfu):=  F(\bfs;\bfgs), \quad F_\sha(\bfu):=  F_\ast\big(\bfq(\bfu)\big), \quad F_\ast(\bfu)= F_\sha\big(\bfp(\bfu)\big),
\end{equation*}
where $F(-)$ can be either $T(-)$,~$T_\calA$, or~$T(-;x)$ or even their partial sums, such as\vspace{-6pt}
\begin{equation*}
T_n(\bfs;\bfgs):=\sum_{\substack{n>n_1>\dots>n_d>0\\ n_j\equiv d-j+1\ppmod{2}}}
\prod_{j=1}^d \frac{\gs_j^{(n_j-d+j-1)/2}}{n_j^{s_j}}.
\end{equation*}

For all words $\bfw\in\fT_2^1$, we set $T_\calA(\bfw):=T_{\calA,\sha}(\bfw)=T_{\calA,\ast}\big(\bfq(\bfw)\big).$ Further, set $\tau(1)=1$ and
\begin{equation*}
\tau(\ty_0^{s_1-1}\ty_1 \dots \ty_0^{s_d-1}\ty_1)=(-1)^{s_1+\dots+s_r} \ty_0^{s_d-1}\ty_1 \dots \ty_0^{s_1-1}\ty_1.
\end{equation*}

\begin{thm}\label{thm:FMTVshuffle}
For all words $\bfw,\bfu\in\fT_1^1$, $\bfv\in\fT_2^1$ and  $s\in\N$, we~have
\begin{itemize}
 \item[{(i)}]
 $T_\calA(\bfu \sha \bfv)=T_\calA(\tau(\bfu)\bfv)$ if $\dep(\bfu)+\dep(\bfv)$ is even;

\item[{(ii)}]
 $T_\calA( (\bfw\bfu) \sha \bfv)=T_\calA(\bfu \sha\tau(\bfw)\bfv)$ if $\dep(\bfu)+\dep(\bfv)+\dep(\bfw)$ is even;

\item[{(iii)}]
 $T_\calA( (\emty_0^{s-1}\emty_1 \bfu) \sha\bfv)
 =(-1)^s T_\calA(\bfu \sha (\emty_0^{s-1}\emty_1\bfv))$ if $\dep(\bfu)+\dep(\bfv)$ is odd.
\end{itemize}
\end{thm}

\begin{proof}
By taking $\bfu=\emptyset$ and then setting $\bfw=\bfu$, we see that (ii) implies (i).
By decomposing $\bfw$ into strings of type $\ty_0^{s-1}\ty_1$,
we see that (iii) implies (ii). So, we only need to prove (iii).

For simplicity, write $\ta=\ty_0$ and $\tb=\ty_1$ for the rest of this proof.
Observe that for any odd prime $p$, the~coefficient of $x^p$ of $T(\bfs;\bfgs;x)$
is nontrivial if and only if $\dep(\bfs)$ is odd. Therefore, if~the depth $d$ of the word
$\bfw$ is even, the coefficient of $x^p$ in $T_\ast\big(\bfq(\tb\bfw);x\big)$ is given as
\begin{equation*}
 {\rm Coeff}_{x^p}\Big[T_\ast\big(\bfq(\tb\bfw);x\big)\Big]=
  {\rm Coeff}_{x^p}\Big[T_\sha\big(\tb\bfw;x\big)\Big]=
  \frac{1}{p} T_{p,\sha}(\bfw)
\end{equation*}
since  $\bfq(\tb\bfw)=\tb\bfq(\bfw)$.
Observe that
\begin{equation*}
\tb \Big( (\ta^{s-1}\tb \bfu) \sha\bfv -(-1)^s  \bfu \sha (\ta^{s-1}\tb \bfv)\Big)
=\sum_{i=0}^{s-1} (-1)^{i} (\ta^{s-1-i}\tb \bfu) \sha (\ta^{i}\tb \bfv).
\end{equation*}
Hence, if~$\dep(\bfu)+\dep(\bfv)$ is odd, then, by first applying $T_\sha(-;x)$ to the above
and then extracting the coefficients of $x^p$ from both sides, we obtain
\begin{align*}
&\, \frac{1}{p}\Big( T_{p,\sha}\big((\ta^{s-1}\tb \bfu) \sha\bfv\big)
  -(-1)^s T_{p,\sha}\big(\bfu \sha (\ta^{s-1}\tb \bfv)\big)\Big)\\
=&\,\sum_{i=0}^{s-1} (-1)^{i} {\rm Coeff}_{x^p}
  \big[T_{\sha}(\ta^{s-1-i}\tb \bfu;x)T_{\sha}(\ta^{i}\tb \bfv;x)\big] \\
=&\,\sum_{i=0}^{s-1} (-1)^{i}\sum_{j=1}^{p-1}
{\rm Coeff}_{x^j}\big[T_{\sha}(\ta^{s-1-i}\tb \bfu;t) \big]\cdot
{\rm Coeff}_{x^{p-j}}\big[T_{\sha}(\ta^{i}\tb \bfv;t)\big]
\end{align*}
from shuffling the product property of the iterated integrals. Now, the last sum
is $p$-integral since $p-j<p$ and $j<p$, and therefore, we obtain
\begin{equation*}
T_{p,\sha}(\ta^{s-1}\tb \bfu) \sha\bfv)\equiv
(-1)^s T_{p,\sha}(\bfu \sha (\ta^{s-1}\tb \bfv)) \pmod{p}
\end{equation*}
which completes the proof of (iii).
\end{proof}

\begin{rem}
In~\cite{Jarossay2014}, Jarossay showed that the corresponding results of Theorem~\ref{thm:FESshuffle}
hold for SMZVs. Theorem \ref{thm:FMTVshuffle}, Conjecture~\ref{conj:KanekoZagierAltVersion} and Proposition~\ref{prop:fESofTS}
clearly imply that similar statements also hold true for SMTVs when the depth conditions
are satisfied, as in Theorem~\ref{thm:FMTVshuffle}. However, it is possible to prove this unconditionally
by using the generalized Drinfeld associator $\Psi_2$ and considering the words of the form
$\tx_0^{s_1-1}(\tx_1+\tx_{-1})\cdots\tx_0^{s_d-1}(\tx_1+\tx_{-1})$ in \cite[Theorem~13.4.1]{ZhaoBook}.
The details of this work will appear in a future paper.
\end{rem}

We can now derive a sum formula for FMTVs.
\begin{thm} \label{thm:sumFMTVs}
Suppose $d\in\N$ is odd. For~all $s_1,\dots,s_d\in\N$, we have\vspace{-6pt}
\begin{equation*}
  T_\calA(1,\bfs)+T_\calA(\bfs,1)+\sum_{j=1}^d\sum_{a=1}^{s_j+1} T_\calA(s_1,\dots,s_{j-1},a,s_j+1-a,s_{j+1},\dots,s_d)=0.
\end{equation*}
\end{thm}

\begin{proof}
This follows immediately from the linear shuffle relation
\begin{equation*}
 T_\calA(  \ty_1   \sha  \ty_0^{s_1-1}\ty_1\dots  \ty_0^{s_d-1}\ty_1 )=-T_\calA(\ty_1\ty_0^{s_1-1}\ty_1\dots  \ty_0^{s_d-1}\ty_1)
\end{equation*}
by taking $s_0=1$ and $\bfu=1$ in Theorem \ref{thm:FMTVshuffle}(iii).
\end{proof}

The following conjecture is supported by all $k\le 9$, numerically.

\begin{conj}\label{conj:T21_k=0}
For all $k\in\N$, we have
\begin{equation*}
  T_\calA(2,\{1\}^k)=\frac{(-1)^{k}}{2^{k-1}}  T_\calA(1,k+1), \quad
  T_\sha^\Sy(\{1\}^w)=\frac{(-1)^{k}}{2^{k-1}}   T_\sha^\Sy(1,k+1).
\end{equation*}
\end{conj}

\begin{prop}
If $k$ is odd, then for all $\ell \le k$, we have
\begin{equation}\label{equ:T121}
T_\calA(\{1\}^\ell,2,\{1\}^{k-\ell})=\frac{(-1)^{\ell}}{\ell+1} \binom{k+1}{\ell}  T_\calA(2,\{1\}^k).
\end{equation}
If, in addition, we assume Conjecture~\ref{conj:T21_k=0} holds, then
\begin{equation}\label{equ:T121withConj}
T_\calA(\{1\}^\ell,2,\{1\}^{k-\ell})=\frac{(-1)^{\ell+k}}{2^{k-1}(\ell+1)} \binom{k+1}{\ell}  T_\calA(1,k+1).
\end{equation}
\end{prop}

\begin{proof}
For all $\ell\le k$, we see the~linear shuffle relations
\begin{equation*}
T_\calA(\ty_1 \sha  \ty_1^\ell \ty_0\ty_1^{k-\ell})=-\ty_1^{\ell+1} \ty_0\ty_1^{k-\ell}.
\end{equation*}
Thus, by setting  $a_\ell=T_\calA(\{1\}^{\ell},2,\{1\}^{k-\ell})$, we obtain
\begin{equation*}
(\ell+2)a_{\ell+1}+(k-\ell+1)a_\ell=0.
\end{equation*}
Hence,
\begin{align*}
 a_{\ell+1}=&\,-\frac{k-\ell+1}{\ell+2}a_\ell=\frac{(k-\ell+1)(k-\ell+2)}{(\ell+2)(\ell+1)}a_{\ell-1}=\cdots\\
 =&\, (-1)^{\ell-1} \frac{(k-\ell+1)(k-\ell+2)\cdots (k+1) }{(\ell+2)(\ell+1)\cdots 2}a_0\\
 =&\, (-1)^{\ell-1} \frac{(k+1)!}{(\ell+2)!(k-\ell)!}a_0 =\frac{(-1)^{\ell-1}}{\ell+2} \binom{k+1}{\ell+1}a_0,
\end{align*}
which yields \eqref{equ:T121}. Then, \eqref{equ:T121withConj} follows immediately if we assume Conjecture~\ref{conj:T21_k=0}.
\end{proof}

\subsection{Values at Small Depths/Weights}
First, we observe that since $\zeta_\calA(s)=0$ for all $s\in\N$ according to~\cite[Theorem~8.2.7]{ZhaoBook} we must have
\begin{align}\label{equ:FESdepth1odd}
S_\calA(s)=-T_\calA(s)=\frac12\zeta_\calA(\bar{s})=&\,
\left\{
 \begin{array}{ll}
  -\tq_2,     \quad  & \quad \hbox{if $s=1$;}\\
  (2^{1-s}-1 )\gb_{s},\quad &\quad  \hbox{if $s\ge 2$,}
 \end{array}
\right.
\end{align}
where $\tq_2$ is the Fermat quotient \eqref{equ:FermatQ2}, and $\gb_{s}$ is given in \eqref{equ:betaw}.
Further, in~depth two, according to~Proposition~2.6 in our arxiv paper 2402.08160, we see that
for all $a,b\in\N$, if~$w=a+b$ is odd, then\vspace{-6pt}
\begin{align}\label{equ:DblFMSTVs}
S_\calA(a,b)=&\,T_\calA(a,b)=\frac{(-1)^a} 2  \Big(1-2^{-w}\Big) \binom{w}{a} \gb_w.
\end{align}

The depth three case is already complicated, and we do not have a general formula.
This is expected since such a formula does not exist for FMZVs. In~the rest of
this section, we will deal with some special~cases.

Next, we prove a proposition that improves a result that Tauraso and the author obtained more than
a decade ago by~applying the newly discovered linear shuffle relations~above.

\begin{prop} \label{prop:wt3FES}
We have\vspace{-8pt}
\begin{align*}
\zeta_\calA(1,1,1)=0,\quad
\zeta_\calA(\bar1,\bar1,\bar1)=&\, -\frac43 \emtq_p^3-\frac{\gb_3}{2} ,  \quad
\zeta_\calA(1,1,\bar1)=\zeta_\calA(\bar1,1,1)=-\frac{\emtq^3}{3}-\frac{7}{8}  \gb_3,
\\
\zeta_\calA(\bar1,1,\bar1)=0,\quad
\zeta_\calA(1,\bar1,1)=&\, \frac{2\emtq^3}{3}+\frac{\gb_3}{4}  , \quad
\zeta_\calA(\bar1,\bar1,1)=-\zeta_\calA(1,\bar1,\bar1)=- \emtq_p^3-\frac{21}{8} \gb_3.
\end{align*}
\end{prop}
\begin{proof}
It immediately follows on from \cite[Propositions~7.3 and 7.6]{TaurasoZh2010} that 
\begin{align*}
\zeta_\calA(\bar1,\bar1,\bar1)=&\, -\frac43 \tq_p^3-\frac{1}{2} \gb_3,  \quad
\zeta_\calA(1,\bar1,1)=-2\zeta_\calA(\bar1,1,1)-\frac{3}{2} \gb_3, \quad
\zeta_\calA(\bar1,1,\bar1)=0, \\
\zeta_\calA(1,1,\bar1)=&\,\zeta_\calA(\bar1,1,1), \quad
\zeta_\calA(\bar1,\bar1,1)=-\zeta_\calA(1,\bar1,\bar1)=- \tq_p^3-\frac{21}{8} \gb_3.
\end{align*}
According to the linear shuffle relations for finite Euler sums, we have
\begin{align*}
-\zeta_\calA(\tb\tc\tc)=\zeta_\calA(\tb\sha \tc\tc)=&\zeta_\calA(\tb\tc\tc)+\zeta_\calA(\tc\tb\tc)+\zeta_\calA(\tc\tc\tb)
\end{align*}
which readily yields the identity
\begin{equation}\label{equ:linShuESbcc}
2\zeta_\calA(1,\bar1,1)+\zeta_\calA(\bar1,\bar1,\bar1)+\zeta_\calA(\bar1,1,\bar1)=0,
\end{equation}
which in turn quickly implies all the evaluations in the proposition.
\end{proof}

\begin{cor}\label{cor:tripFMTV}
We have
\begin{equation*}
T_\calA(1,1,1)=-S_\calA(1,1,1)=\frac{3}{16} \gb_3.
\end{equation*}
\end{cor}

\begin{proof}
The corollary is an immediate consequence of the definitions using  Proposition \ref{prop:wt3FES}.
Alternatively, we can prove it directly, as follows: since $\zeta_\calA(1,\bar1,\bar1)=-\zeta_\calA(\bar1,\bar1,1)$
from a reversal, and $\zeta_\calA(1,1,1)=0$, we obtain
\begin{alignat*}{4}
8T_\calA(1,1,1)=&\,\zeta_\calA(\bar1,1,\bar1)+\zeta_\calA(\bar1,\bar1,\bar1)+\zeta_\calA(1,\bar1,1)
-&&\zeta_\calA(1,1,\bar1)-\zeta_\calA(\bar1,1,1)\\
=&\,-\zeta_\calA(1,\bar1,1)-2\zeta_\calA(1,1,\bar1)    \quad && (\text{by \eqref{equ:linShuESbcc}})\\
=&\,\zeta_\calA(\bar2,1)+\zeta_\calA(\bar1,2)-\zeta_\calA(1)\zeta_\calA(1,\bar1)   \quad && (\text{by shuffle})\\
=&\,\frac{3}{2} \gb_3.
\end{alignat*}
according to  \cite[Theorem~8.6.4]{ZhaoBook}.
\end{proof}

\begin{prop}\label{prop:tripSMTV}
We have
\begin{equation*}
T_\sha^\Sy(1,1,1)=-S_\sha^\Sy(1,1,1)=\frac{3}{16} \zeta(3).
\end{equation*}
\end{prop}

\begin{proof}
The weighted three Euler sums are all expressible in terms of $\zeta(\bar2,1),$  $\zeta(\bar1,1,1)$ and
$\zeta(\bar1,2)$ by \cite[Proposition~14.2.7]{ZhaoBook}. Hence, one easily deduces that\vspace{-6pt}
\begin{align*}
\zeta_\sha^\Sy(1,1,1)=&\, \zeta_\sha^\Sy(\bar1,1,\bar1)=0,\\
\zeta_\sha^\Sy(\bar1,1,1)=&\, \zeta_\sha^\Sy(1,1,\bar1)= \zeta(\bar1,1,1)+\zeta(\bar1)\zeta_\sha(1,1)-\zeta_\sha(1,\bar1)\zeta_\sha(1)+\zeta_\sha(1,1,\bar1)\\
=&\, \zeta(\bar1,1,1)+\zeta(\bar1)\frac{T^2}2-\Big(\zeta(\bar1)T-\zeta(\bar1,\bar1)\Big)T
  +\zeta(\bar1)\frac{T^2}2-\zeta(\bar1,\bar1)T+\zeta(\bar1,\bar1,1)\\
=&\, \zeta(\bar1,1,1)+\zeta(\bar1,\bar1,1),\\
\zeta_\sha^\Sy(\bar1,\bar1,1)=&\, -\zeta_\sha^\Sy(1,\bar1,\bar1)
%= \zeta(\bar1,\bar1,1)+\zeta(\bar1)\zeta(\bar1,1)+\zeta(\bar1,\bar1)\zeta_\sha(1)-\zeta_\sha(1,\bar1,\bar1),\\
%=&\, \zeta(\bar1,\bar1,1)+\zeta(\bar1)\zeta(\bar1,1)+\zeta_\sha(\bar1,\bar1)T-\zeta(\bar1,\bar1)T+2\zeta(\bar1,\bar1,1)
=3\zeta(\bar1,\bar1,1)+3\zeta(\bar1,1,1) ,\\
\zeta_\sha^\Sy(1,\bar1,1)=&\,2\zeta_\sha(1,\bar1,1)-2\zeta(\bar1,1)\zeta_\sha(1)
    =-2\zeta(\bar1,\bar1,\bar1)-2\zeta(\bar1,1,\bar1),\\
\zeta_\sha^\Sy(\bar1,\bar1,\bar1)=&\,2\zeta(\bar1,\bar1,\bar1)+2\zeta(\bar1)\zeta(\bar1,\bar1)
=4\zeta(\bar1,\bar1,\bar1)+4\zeta(\bar1,1,\bar1).
\end{align*}

In \cite[Proposition~14.2.7]{ZhaoBook}, we have
\begin{align*}
\zeta(3)=&\,  8\zeta(\bar2,1), \quad
\zeta(\bar1,\bar1,1)= \zeta(\bar1,2)-5\zeta(\bar2,1)+\zeta(\bar1,1,1), \\
 \zeta(\bar1,1,\bar1)=&\, \zeta(\bar2,1)+\zeta(\bar1,1,1),\quad
\zeta(\bar1,\bar1,\bar1)= \zeta(\bar1,2)+\zeta(\bar1,1,1).
\end{align*}
Thus, we get
\begin{align*}
T_\sha^\Sy(1,1,1)=&\, \frac14\Big(\zeta(\bar1,\bar1,\bar1)+\zeta(\bar1,1,\bar1)- \zeta(\bar1,1,1)-\zeta(\bar1,\bar1,1)\Big)
=\frac64\zeta(\bar2,1)=\frac3{16}\zeta(3),\\
S_\sha^\Sy(1,1,1)=&\, -\frac14\Big(\zeta(\bar1,\bar1,\bar1)+\zeta(\bar1,1,\bar1)- \zeta(\bar1,1,1)-\zeta(\bar1,\bar1,1)\Big)
=-\frac3{16}\zeta(3),
\end{align*}
as desired.
\end{proof}

In general, we can use linear shuffles to derive many relations from the finite Euler sums. For~example,
\begin{align*}
%\tb\sha \tb\ta\tc: &\, 3\zeta_\calA(1,1,\bar2)+\zeta_\calA(1,2,\bar1)+\zeta_\calA(1,\bar2,\bar1)=0,\\
%\tb\sha \tc\ta\tb: &\, 2\zeta_\calA(1,\bar1,\bar2)+\zeta_\calA(\bar1,\bar1,2)+2\zeta_\calA(\bar1,\bar2,1)=0,\\
%\tb\sha \ta\tb\tc: &\, 2\zeta_\calA(1,2,\bar1)+2\zeta_\calA(2,1,\bar1)+\zeta_\calA(2,\bar1,\bar1)=0,\\
\tb\sha \ta\tc\tb: &\, 2\zeta_\calA(1,\bar2,\bar1)+\zeta_\calA(2,\bar1,\bar1)+2\zeta_\calA(\bar2,\bar1,1)=0,\\
\tb\sha \ta\tc\tc: &\, 2\zeta_\calA(1,\bar2,1)+\zeta_\calA(2,\bar1,1)+\zeta_\calA(\bar2,\bar1,\bar1)+\zeta_\calA(\bar2,1,\bar1)=0,\\
\ta\tb\sha \tb\tc: &\, 3\zeta_\calA(2,1,\bar1)+\zeta_\calA(2,\bar1,\bar1)+\zeta_\calA(1,\bar2,\bar1)+\zeta_\calA(1,2,\bar1)+\zeta_\calA(1,\bar1,\bar2)=0,\\
\ta\tb\sha \tc\tb: &\, 2\zeta_\calA(2,\bar1,\bar1)+2\zeta_\calA(\bar2,\bar1,1)+2\zeta_\calA(\bar1,\bar2,1)+\zeta_\calA(\bar1,\bar1,2)=0,\\
\ta\tb\sha \tc\tc: &\, 2\zeta_\calA(2,\bar1,1)+\zeta_\calA(\bar2,\bar1,\bar1)+\zeta_\calA(\bar2,1,\bar1)+\zeta_\calA(\bar1,\bar2,\bar1)
+\zeta_\calA(\bar1,2,\bar1)+\zeta_\calA(\bar1,1,\bar2) =0,\\
\hspace{-1pt}\tb\sha \tb\ta\tc^2: &\, 3\zeta_\calA(1,1,\bar2,1)+\zeta_\calA(1,2,\bar1,1)+\zeta_\calA(1,\bar2,\bar1,\bar1)+\zeta_\calA(1,\bar2,1,\bar1)=0,\\
\tb\sha \tc^4: &\, 2\zeta_\calA(1,\bar1,1^3)+\zeta_\calA(\bar1^3,1,1)+\zeta_\calA(\bar1,1,\bar1^2,1)
+\zeta_\calA(\bar1,1^2,\bar1^2)+\zeta_\calA(\bar1,1^3,\bar1)=0.
\end{align*}

We can also use reversal and shuffle relations to express all finite Euler sums of weight up to 6
according to the explicitly given basis in each weight. Aided by Maple computation, we arrive at the following
main theorem on the structure of finite Euler sums of a lower~weight.

\begin{thm}\label{thm:FESbasis}
Let $\FES_w$ be the $\Q$-vector space generated by finite Euler sums of weight $w$. Then,
we have the following generating sets for $w<7$:
\begin{align*}
\FES_1=&\, \langle \emtq_2\rangle,  \quad
\FES_2=  \langle \emtq_2^2\rangle,\quad
\FES_3=  \langle \emtq_2^3,\gb_3\rangle,\quad
\FES_4=  \langle \emtq_2^4, \emtq_2\gb_3, \zeta_\calA(1,\bar3) \rangle,\\
\FES_5=&\, \langle \emtq_2^5, \emtq_2^2\gb_3, \gb_5, \zeta_\calA(\bar1,2,2), \zeta_\calA(\bar1,\bar2,2) \rangle, \\
\FES_6=&\, \langle \emtq_2^6, \emtq_2^3\gb_3, \gb_3^2, \emtq_2\gb_5, \zeta_\calA(\bar1,1,2,2),\zeta_\calA(\bar1,2,2,1),\zeta_\calA(\bar1,2,1,2),\zeta_\calA(\bar1,\{1\}^3,2) \rangle.
\end{align*}
\end{thm}

Let $\{F_k\}_{k\ge 0}$ be the Fibonnacci sequence defined by $F_0=F_1=1$
and $F_k=F_{k-1}+F_{k-2}$ for all $k\ge 2$. Then, Theorem~\ref{thm:FESbasis} provides strong support
for the next~conjecture.

\begin{conj}\label{conj:FESbasis}
For every positive integer $w$, the $\Q$-space $\FES_w$ has the following basis:
\begin{equation*}
\Big\{\zeta_\calA(\bar1,b_2,\dots,b_d): d\ge 0, b_j=1 \text{ or } 2, 1+b_2+\dots+b_d=w\Big\}.
\end{equation*}
Consequently, $\dim_\Q \FES_w=F_{w-1}$ for all $w\ge 1$.
\end{conj}

One may compare this to the conjecture on the ordinary Euler sums proposed by Zlobin \cite[Conjecture~14.2.3]{ZhaoBook}.

\begin{conj}\label{conj:Zlobin}
For every positive integer $w$ the $\Q$-space $\ES_w$ has the following basis:
\begin{equation*}
     \Big\{\zeta(\ol{b_1},b_2,\dots,b_d): d\ge 1, b_j=1 \text{ or }2, b_1+b_2+\dots+b_d=w\Big\}.
\end{equation*}
Consequently, $\dim_\Q \ES_w=F_w$ for all $w\ge 1$.
\end{conj}

Theorem~\ref{thm:FESbasis} implies that the set in Conjecture \ref{conj:FESbasis}
is a generating set for all $w<7$ since
\begin{align*}
\zeta_\calA(\bar1,1)=&\, -2\tq_2, \quad
\zeta_\calA(\bar1,1)=  \tq_2^2,\quad
\zeta_\calA(\bar1,2)=  \tfrac34 \beta_3,\quad
\zeta([1, \bar1,1)=  \tfrac24 \tq_2^3+\tfrac14 \beta_3,\\
\zeta_\calA(\bar1,1,2)=&\, \tfrac94 \tq_2 \beta_3-\zeta_\calA(1,\bar3),\quad
\zeta_\calA(\bar1,\{1\}^3)= \tfrac1{12} \tq_2^4+\tfrac78 \tq_2 \beta_3+\tfrac14 \zeta_\calA(1,\bar3),\\
\zeta_\calA(\bar1,2,1)=&\, \tfrac12 \zeta_\calA(1,\bar3) -\tfrac{12}4 \tq_2 \beta_3,\\
\zeta_\calA(\bar1,2,1,1)=&\, \tfrac{695}{128}  \beta_5-\tfrac54 \zeta_\calA(\bar1,2,2)-2 \zeta_\calA(\bar1,1,1,2)-\tfrac94 \tq_2^2 \beta_3,\\
\zeta_\calA(\bar\{1\}^4,1)=&\, -\tfrac{1}{60}\tq_2^5-\tfrac{23}{24} \tq_2^2 \beta_3-\tfrac{1}{8} \zeta_\calA(\bar1,2,2)-\tfrac{1}{2} \zeta_\calA(\bar1,1,1,2)-\tfrac{25}{256} \beta_5,\\
\zeta_\calA(\bar1,1,2,1)=&\, \tfrac{33}{8} \tq_2^2 \beta_3-\tfrac{555}{128}  \beta_5+\tfrac{5}{4} \zeta_\calA(\bar1,2,2)+2 \zeta_\calA(\bar1,1,1,2),\\
\zeta_\calA(\bar1,1,2,1,1)=&\, -\tfrac{1}{2} A+2B+C+D  +\tfrac{9}{4} \beta_3^2+\tfrac{5}{8} \tq_2^3 \beta_3+\tfrac{205}{64} \tq_2 \beta_5,\\
\zeta_\calA(\bar\{1\}^4,2)=&\, -\tfrac{3}{4} A+\tfrac{19}{8} B+\tfrac14 C+D +\tfrac{201}{32} \beta_3^2+\tq_2^3 \beta_3-\tfrac{645}{256} \tq_2 \beta_5,\\
\zeta_\calA(\bar1,2,\{1\}^3)=&\, \tfrac12 A-\tfrac{19}{8} B-\tfrac54 C-2 D  -\tfrac{1113}{256} \beta_3^2-\tfrac54 \tq_2^3 \beta_3-\tfrac{1685}{256} \tq_2 \beta_5,\\
\zeta_\calA(\bar1,\{1\}^5)=&\,\tfrac14 A-\tfrac{13}{16} B-\tfrac{1}{8} C-\tfrac{1}{2} D -\tfrac16\tq_2^3 \beta_3+\tfrac{817}{512} \tq_2 \beta_5-\tfrac{811}{512} \beta_3^2+\tfrac{1}{360} \tq_2^6,
\end{align*}
where $A=\zeta_\calA(\bar1,1,2,2), B=\zeta_\calA(\bar1,2,1,2), C=\zeta_\calA(\bar1,2,1,2),$ and $D=\zeta_\calA(\bar1,2,2,1)$.

By using the evaluations of finite Euler sums, we can find all FMTVs of weight less than 7. For~example, we have\vspace{-6pt}
\begin{align*}
T_\calA(1,1,2)=&\,-\frac18\zeta_\calA(1,\bar3)-\frac{21}{16} \tq_2\gb_3,\\
T_\calA(1,2,2)=&\,-\frac{1605}{256}\gb_5+\tfrac{9}{2} \tq_2^2 \beta_3+3\zeta_\calA(\bar1,1,1,2).
\end{align*}
We then have the following structural theorem for these~FMTVs:

\begin{thm}\label{thm:FMTVbasis}
Let $\FMTV_w$ be the $\Q$-vector space generated by FMTVs of weight $w$. Then,
we have the following generating sets for $w<7$:
\begin{align*}
\FMTV_1=&\, \langle \emtq_2\rangle,  \quad
\FMTV_2=  \langle 0 \rangle,\quad
\FMTV_3=  \langle \gb_3\rangle,\quad
\FMTV_4=  \langle \emtq_2\gb_3, \zeta_\calA(1,\bar3) \rangle,\\
\FMTV_5=&\, \langle \gb_5, \zeta_\calA(\bar1,2,2), \zeta_\calA(\bar1,1,1,2) \rangle, \quad
\FMTV_6= \langle \gb_3^2, \emtq_2\gb_5, \zeta_\calA(\bar1,2,1,2) \rangle.
\end{align*}
\end{thm}
Moreover, by~using numerical computation aided by Maple (see \cite[Appendix D]{ZhaoBook},
for the pseudo codes), we can find a generating set of $\FMTV_w$ for every $w\le 13.$
We will list the corresponding dimensions at the end of this~paper.

\subsection{Homogeneous~Cases}
In this subsection, we will compute finite Euler sums $\zeta(\bfs)$ when $\bfs$ is homogeneous,
i.e., $\bfs=(\{s\}^d)$ for some $s\in\db$. Then, we will consider the corresponding results
for~FMTVs.

\begin{prop}
Let $\N_{\text{odd}}$ be the set of odd positive integers. For~any $d, s\in\N$, we have
\begin{equation*}
\zeta_\calA(\{\bar{s}\}^d)\in \sum_{\substack{k_0\in \N,\, k_1,\dots,k_\ell\in \N_{\text{odd}}  \\  \gd_{s,1}k_0+k_1+\dots+k_\ell=d}}
\emtq_2^{\gd_{s,1}k_0} \gb_{sk_j} \cdots \gb_{sk_j} \Q,
\end{equation*}
where $\gd_{s,1}$ is the Kronecker symbol. In~particular, $\zeta_\calA(\{\bar{s}\}^d)=0$ for all even $s$.
\end{prop}

\begin{proof}
Let $\Pi=(P_1,\dots,P_\ell)\in [d]$ denote any partition of $(1,\dots,d)$ into odd parts, i.e.,~all
of $|P_j|$'s are odd numbers, where $|P_j|$ is the cardinality of the set $P_j$. Let
\begin{equation*}
\calC(\Pi)=(-1)^{d-\ell} (|P_1|-1)!\cdots (|P_\ell|-1)!.
\end{equation*}

Observe that $\zeta_\calA(\bar{n})=\zeta_\calA(n)=0$ if $n$ is even.
Then, it follows easily from \cite[(18)]{Hoffman2004b} that
\begin{align*}
\zeta_\calA(\bar{s})= \sum_{\Pi=(P_1,\dots,P_\ell) \in [d]} \calC(\Pi) \zeta_\calA\Big(\ol{s|P_1|}\Big)\cdots \zeta_\calA\Big(\ol{s|P_\ell|}\Big).
\end{align*}

The proposition follows from \eqref{equ:FESdepth1odd} immediately.
\end{proof}

\begin{eg}
There are many ways to partition $6$ elements, say $\{a_1,\dots,a_6\}$ into odd parts:
one way to get $(\{1\}^6)$, $\binom{6}{5}$ ways to obtain $(1,5)$ (e.g., $\{a_2\},\{a_1,a_3,\dots,a_6\}$),
$\binom{6}{3}/2$ ways to obtain $(3,3)$, and~ $\binom{6}{3}$ ways to obtain $(1,1,1,3)$. Hence,
\begin{equation*}
\zeta_\calA(\{\bar1\}^6)=\tfrac{4}{45}\tq_2^6+\tfrac{3}{4}\tq_2\gb_5+\tfrac{1}{8}\gb_3^2+\tfrac{2}{3}\tq_2^3\gb_3
\end{equation*}
when using the formula in \eqref{equ:FESdepth1odd}.
We would like to point out that the term $3 q_p B_{p-5}/20$ (corresponding to the second term $\tfrac{3}{4}\tq_2\gb_5$
on the right-hand side above) was accidentally dropped from the right-hand side of \cite[(36)]{TaurasoZh2010}.
\end{eg}

One may compare the next corollary with the well-known result that $\zeta_\calA(\{1\}^d)=0$ for all $d\in\N$ (see, e.g.,~\cite[Theorem~8.5.1]{ZhaoBook}).

\begin{prop}\label{prop:FMTV2d=0}
For all $d\in\N$, we have
\begin{equation*}
  T_\calA(\{1\}^{2d})=0.
\end{equation*}
\end{prop}

\begin{proof}
Taking $\bfs=(\{1\}^{2d-1})$ in Theorem~\ref{thm:sumFMTVs} yields the proposition at once.
\end{proof}

We now derive the symmetric MTV version of Proposition~ \ref{prop:FMTV2d=0}.
\begin{prop}\label{prop:SMTV2d=0}
For all $d\in\N$, we have
\begin{equation*}
  T_\sha^\Sy(\{1\}^{2d})=0.
\end{equation*}
\end{prop}

\begin{proof}
For any $\ell\in\N$, we have the relation for the regularized value (see, e.g.,~\cite[Section 2]{IKZ2006})
\begin{equation*}
\int_0^\eps \bigg(\frac{dt}{1-t^2}\bigg)^\ell
=\frac{1}{\ell!} \bigg( \int_0^\eps\frac{dt}{1-t^2}\bigg)^\ell
=\frac{1}{\ell!} \bigg( \frac12\int_0^\eps\bigg(\frac{dt}{1-t}+\frac{dt}{1+t}\bigg)\bigg)^\ell,
\end{equation*}
which implies that
\begin{equation*}
T_\sha(\{1\}^\ell)=\frac{1}{\ell!2^\ell} \Big(\zeta_\sha(1)+\log 2\Big)^\ell.
\end{equation*}
According to the definition,
\begin{equation*}
 T_\sha^\Sy(\{1\}^{2d})=\sum_{i=0}^{2d}  (-1)^i   T_\sha(\{1\}^i)T_\sha(\{1\}^{2d-i})
 = \frac{1}{2^{2d}}\sum_{i=0}^{2d} \frac{(-1)^i}{\ell!(2d-\ell)!} \Big(\zeta_\sha(1)+\log 2\Big)^{2d}=0
\end{equation*}
as desired.
\end{proof}

By conducting extensive numerical experiments, we found that the following relations must be valid.
\begin{conj}\label{conj:T1_w}
For all odd $w\in\N$, we have
\begin{equation*}
  T_\calA(\{1\}^w)=-S_\calA(\{1\}^w)=\frac{2^{w-1}-1}{2^{2w-2}} \beta_w, \quad
  T_\sha^\Sy(\{1\}^w)=-S_\sha^\Sy(\{1\}^w)=\frac{2^{w-1}-1}{2^{2w-2}} \zeta(w).
\end{equation*}
\end{conj}
The conjecture holds when $w=3$ according to Corollary~\ref{cor:tripFMTV} and Proposition~\ref{prop:tripSMTV}. Aided by Maple, we can also
rigorously prove the conjecture for $w=5$ and $w=7$ by using the tables of values of finite Euler sums produced by reversal, shuffle, and
linear shuffle relations, and~the table of values for Euler sums is available online~\cite{BlumleinBrVe2010}.

Moreover, Conjecture~\ref{conj:T1_w}
still holds true for $T_\calA(\{1\}^w)=T_\sha^\Sy(\{1\}^w)=0$ when $w$ is even because of Propositions~\ref{prop:FMTV2d=0} and \ref{prop:SMTV2d=0}. However,~for $S$-values, we have
another conjecture.
\begin{conj}\label{conj:S1_w}
For all even $w\in\N$, there are rational numbers $c_{j}\in\Q$, $1\le j\le w/2$, such that
\begin{equation*}
  S_\calA(\{1\}^w)=\sum_{j=1}^{w/2} c_{j}S_\calA(j,w-j), \quad
  S_\sha^\Sy(\{1\}^w)=\sum_{j=1}^{w/2} c_{j}S_\sha^\Sy(j,w-j).
\end{equation*}
Moreover, $S_\calA(j,w-j)$ and $1\le j\le w/2$ are $\Q$-linearly independent,
and~$S_\sha^\Sy(j,w-j)$ and $1\le j\le w/2$ are $\Q$-linearly independent.
\end{conj}
Note that $S_\calA(j,w-j)\in\calA$ when $S_\sha^\Sy(j,w-j)$ are all real~numbers.

\section{Alternating Multiple $T$-Values}\label{s4}
We now turn to the alternating version of MTVs and derive some relations among them.
These values are intimately related to the colored MZVs of level 4
(i.e., multiple polylogarithms evaluated for the fourth roots of unity). We refer the interested reader to~\cite{XuZhao2020a,XuZhao2020c}
for the fundamental results concerning these~values.

Recall that for any $(\bfs,\bfgs)\in\N^d\times \{\pm1\}^d$ ,
we have defined the finite alternating multiple $T$-values as\vspace{-8pt}
\begin{equation*}
T(\bfs;\bfgs):=\bigg(\sum_{\substack{p>n_1>\dots>n_d>0\\ n_j\equiv d-j+1\ppmod{2}}}
        \prod_{j=1}^d \frac{\gs_j^{(n_j-d+j-1)/2}}{n_j^{s_j}} \bigg)_{p\in\calP}\in\calA.
\end{equation*}
We have seen from Theorem~\ref{thm:FMTVshuffle} in Section~\ref{sec:linShuffle} that these values satisfy the
linear shuffle relations. It is also not hard to get the reversal relations when the depth is even, as~shown~below.

\begin{prop}[{Reversal relations of finite alternating MTVs}]\label{prop:FAMTVreversal}
Let $\bfs\in\N^d$ for some even $d\in\N$. Then,
\begin{align}\label{equ:FAMTSVReversalEvenDep}
T_\calA(\ola{\bfs},\ola{\bfgs})=&\, (\gs_1,\dots,\gs_d)^{(p-1-d)/2} (-1)^{|\bfs|} T_\calA(\bfs,\bfgs),
\end{align}
where the element $(-1)^{(p-1-d)/2}=\big( (-1)^{(p-1-d)/2}  \pmod{p}\big)_{3\le\in\calP}\in\calA$.
\end{prop}

\begin{proof}  Let $p$ be an odd prime. Then, by changing the indices $n_j\to p-n_j$, we obtain
\begin{align*}
T_p(\bfs,\bfgs):=&\,\sum_{\substack{p>n_1>\dots>n_d>0\\ n_j\equiv d-j+1\ppmod{2}}}
        \prod_{j=1}^d \frac{\gs_j^{(n_j-d+j-1)/2}}{n_j^{s_j}} \\
\equiv &\,(-1)^{|\bfs|} \sum_{\substack{p>n_d>\dots>n_1>0\\ p-n_j\equiv d-j+1\ppmod{2}}}
        \prod_{j=1}^d \frac{\gs_j^{(n_j-p+d-j+1)/2}}{n_j^{s_j}} \pmod{p}.
\end{align*}

Let $t_j=s_{d+1-j}$, $\eps_j=\gs_{d+1-j}$, and~$k_j=n_{d+1-j}$. Then, by changing the indices, we obtain $j\to d+1-j$ (since $d$ is even)\vspace{-6pt}
\begin{align*}
T_p(\bfs,\bfgs) \equiv &\,(-1)^{|\bfs|} \sum_{\substack{p>n_1>\dots>n_d>\\ p-k_j\equiv j\ppmod{2}}}
        \prod_{j=1}^d \frac{\eps_j^{(k_j-p+j)/2}}{k_j^{t_j}}\\
\equiv &\,(-1)^{|\bfs|} \sum_{\substack{p>n_1>\dots>n_d>\\ p-k_j\equiv j\ppmod{2}}}
        (\gs_1\cdots\gs_d)^{(d-p+1)/2} \prod_{j=1}^d \frac{\eps_j^{(k_j-d+j-1)/2}}{k_j^{t_j}}\\
\equiv &\,(\gs_1\cdots\gs_d)^{(d-p+1)/2} (-1)^{|\bfs|} \sum_{\substack{p>k_1>\dots>k_d>\\ k_j\equiv d-j+1\ppmod{2}}}
         \prod_{j=1}^d \frac{\eps_j^{(k_j-d+j-1)/2}}{k_j^{t_j}}\\
\equiv &\,(\gs_1\cdots\gs_d)^{(d-p+1)/2} (-1)^{|\bfs|} T_p(\bft,\bfeps) \\
\equiv &\, (\gs_1\cdots\gs_d)^{(d-p+1)/2}(-1)^{|\bfs|} T_p(\ola{\bfs},\ola{\bfgs})
\end{align*}
as desired.
\end{proof}

It should be clear to the attentive reader that $T$-values are always intimately related to the $S$-values
when the depth is odd because of the reversal relations. Even though we did not consider this in the
above, it plays a key role in the proof of the next~result.

\begin{prop} \label{prop:FAMTVwt1}
Let $q_2(p)=(2^{p-1}-1)/2$ for all $p>2$. Then, we have
\begin{align*}
S_\calA(\bar1)=-\emtq_2/2, \quad T_\calA(\bar1)= \Big((-1)^{\tfrac{p-1}2} q_2(p)/2 \pmod{p}\Big)_{p>2}\in\calA.
\end{align*}
\end{prop}
\begin{proof} Recall that\vspace{-6pt}
\begin{equation*}
S_p(1):=\sum_{p>k>0, 2|k} \frac1k, \quad S_p(\bar1):=\sum_{p>k>0, 2|k}  \frac{(-1)^{k/2}}{k}.
\end{equation*}

According to \cite[Theorem~3.2]{Sun2008}, we see that
\begin{equation*}
S_p(1)+S_p(\bar1)=\sum_{p>k>0, 2|k} \bigg(\frac1k+\frac{(-1)^{k/2}}{k}\bigg)=\sum_{p>k>0, 4|k} \frac2k\equiv -\frac32 q_p(2) \pmod{p}.
\end{equation*}
Since $S_p(1)=\zeta_p(\bar1)/2=-q_p(2)$, we immediately see that $S_\calA(\bar1)=-\tq_2/2$.
By taking the reversal, we obtain
\begin{align*}
T_p(\bar1)=\sum_{p>k>0, 2\nmid k} \frac{(-1)^{(k-1)/2}}{k}
=&\, \sum_{p>k>0, 2|k}  \frac{(-1)^{(p-k-1)/2}}{p-k} \\
\equiv&\, - (-1)^{\tfrac{p-1}2} S_p(\bar1) \equiv (-1)^{\tfrac{p-1}2}\frac{q_p(2)}2 \pmod{p},
\end{align*}
as desired.
\end{proof}

As we analyzed in \cite[p.~239]{ZhaoBook}, there is overwhelming evidence that $\tq_2\ne 0$
in $\calA$. In~\cite[Theorem~1]{Silverman1998}, Silverman even showed that if~abc-conjecture holds, then
\begin{equation*}
    \Big|\Big\{p\le X: q_2(p) \ne 0 \pmod{p}\Big\} \Big|=O(\log(X) )  \quad\text{as}\quad X\to\infty.
\end{equation*}
In we are sure the following conjecture is~true.

\begin{conj} \label{conj:q2Conj}
For every pair of positive integers $m>a>0$, $\gcd(m,a)=1$,
there are infinitely many primes $p\equiv a\pmod{m}$
such that $q_2(p)\not\equiv 0\pmod{p}$.
\end{conj}

\begin{thm} \label{thm:abcConj}
If Conjecture~\ref{conj:q2Conj} holds for $m=4$, then $T_\calA(1)$ and $T_\calA(\bar1)$
are $\Q$-linearly independent.
\end{thm}
\begin{proof}
If $c_1 T_\calA(1)+c_2T_\calA(\bar1)=0$ in $\calA$ for some $c_1,c_2\in\Q$, then, according to Proposition~\ref{prop:FAMTVwt1},
we see that $(c_1+c_2)q_p(2)\equiv 0 \pmod{p}$ for infinitely many primes $p\equiv 1 \pmod{4}$.
If Conjecture~\ref{conj:q2Conj} holds the form $m=4$, then $c_1+c_2\equiv 0 \pmod{p}$
for infinitely many primes $p\equiv 1 \pmod{4}$. This would force $c_1+c_2=0$.
A similar consideration for primes $p\equiv 3 \pmod{4}$ implies that  $c_1-c_2=0$.
Hence, we must have $c_1=c_2=0$, which shows that $T_\calA(1)$ and $T_\calA(\bar1)$
are $\Q$-linearly independent.
\end{proof}

Define the \emph{finite Catalan's constant} as
\begin{equation*}
    G_\calA:=\Big(\frac{E_{p-3}}2 \Big)_{3<p\in\calP} \in\calA.
\end{equation*}

\begin{prop}
Let $\FAMTV_w$ be the vector space generated by finite alternating MTVs over $\Q$.
We have the following generating sets of $\FAMTV_w$ for $w<3$:
\begin{align*}
\FAMTV_1=\langle \emtq_2, (-1)^{p'}\emtq_2  \rangle, \quad
\FAMTV_2=\langle G_\calA, (-1)^{p'}G_\calA \rangle.
\end{align*}
\end{prop}
\begin{proof}
The $w=1$ case is trivial. For~$w=2$, we already know $T_\calA(1,1)=T_\calA(2)=0$ from Theorem~\ref{thm:FMTVbasis}.
Let $\ta=\ty_0$, $\tb=\ty_1$, and $\tc=\ty_{-1}$ in the rest of the proof.
For alternating values, we first have the linear shuffle relation
\begin{equation*}
T_\calA(\tb\sha \tc)=-T_\calA(\tb\tc)\Rightarrow 2T_\calA(\tb\tc)+T_\calA(\tc\tb) \Rightarrow 2T_\calA(1,\bar1)+T_\calA(\bar1,\bar1)=0.
\end{equation*}

By using complicated computation (see Proposition~4.4 of our arxiv paper 2402.08160 and notice \eqref{equ:oldNewConvert}),
we have the additional relation
\begin{equation*}
T_\calA(\bar2)=G_\calA=-2T_\calA(1,\bar1).
\end{equation*}
Then, from the reversal relation \eqref{equ:FAMTSVReversalEvenDep}, we easily see that $T_\calA(\bar1,1)=-(-1)^{p'}T_\calA(1,\bar1).$
This completes the proof of the proposition.
\end{proof}

\section{Dimensions of $\FMTV$ and $\AMTV$}
We first need to point out that it is possible to study the alternating MTVs by converting them into
colored MZVs of level 4 and then applying the setup in~\cite{SingerZ2019}. For~example,
\begin{align*}
T(\bar2,\bar3)=&\, \sum_{n_1>n_2>0} \frac{(-1)^{n_1-2}(-1)^{n_2-1}}{(2n_1-2)^2(2n_2-1)^3} \\
=&\,\sum_{k_1>k_2>0} \frac{i^{k_1}(1+(-1)^{k_1})i^{k_2-1}(1-(-1)^{k_2})}{k_1^2 k_2^3} \\
=&\, -i\Big( Li_{2,3}(i,i)+Li_{2,3}(i,-i)-Li_{2,3}(-i,i)-Li_{2,3}(-i,-i)\Big).
\end{align*}
The caveat is that we need to extend our scalars to $\Q[i]$ in general. At~the end of~\cite{SingerZ2019},
we observed that $\dim_\Q \FCMZV^{4}_w\le 2^w$ for all $w\ge 1$, where $\FCMZV^{4}_w$ is the space spanned by all
colored MVZ of level 4 and weight $w$ over $\Q$.  By~the following, we expect that the
\begin{equation*}
\dim_\Q \FAMTV_w\le \dim_\Q \FAMMV_w \le 2^w,
\end{equation*}
where $\FAMMV$ is the space spanned by all the finite multiple mixed values.
Here, according to~\cite{XuZhao2020a}, the~multiple mixed values mean we allow
all possible even/odd combinations in the definition of such series instead of a fixed pattern,
such as that which appears in MTVs and MSVs).

\begin{cpp} \label{cpp}
Let $S$ be a set of colored MZVs (including MZVs and Euler sums) or (alternating) multiple mixed values
(or their variations/analogs, such as finite, symmetric, interpolated versions, etc.).
Then, the following statements should~hold.
\begin{enumerate}
  \item[(1)] Suppose all elements in $S$ have the same weight. If~they are linearly independent over $\Q$,
        then they are algebraically independent over $\Q$.
  \item[(2)]  If the weights of the values in $S$ are all different, then the values are linearly independent over $\Q$
(but, of course, may not be algebraically independent over $\Q$).
  \item[(3)]  If there is only one nonzero element in $S$, then it is transcendental over $\Q$.
\end{enumerate}
\end{cpp}

For example, we expect that $\zeta(n)$'s are not only irrational but are also transcendental for all $n\ge 2$.
We also expect that $\tq_2$ and $\gb_k$ are transcendental for all odd $k\ge 3$
and are all algebraically independent over $\Q$

Recall that $\MTV_w$ (resp. $\FMTV_w$) is the $\Q$-vector space generated by MTVs (resp. finite MTVs) of weight $w$.
Similarly,  we denote by $\AMTV$ (resp. $\FAMTV$) the space generated by alternating MTVs 
(resp. finite alternating MTVs) of weight $w$.
From numerical computation, we conjecture the following upper bounds for the dimensions of $\FMTV_w$ and $\FAMTV_w$.
In order to compare to the classical case, we tabulate the results together in the following. The~main software we used was the open source computer algebra system~GP-Pari.

\begin{center}
\begin{tabular}{|c|c|c|c|c|c|c|c|c|c|c|c|c|c|c|c|} \hline
       $\textbf{w}$         & \textbf{0}& \textbf{1} &  \emph{2}  &  \textbf{3}  &  \textbf{4}  &  \textbf{5}  &  \textbf{6}  & \textbf{7}  &  \textbf{8}  &  \textbf{9}  &  \textbf{10} &  \textbf{11}  &  \textbf{12} &  \textbf{13}  \\ \hline
$\FMTV_w$  & 0& 1 &  0  &  1  &  2  &  3  &  3  & 6  &  9  &  15 &  17 &  32  &  44 &  76  \\ \hline
$\MTV_w$  & 1 & 0 &  1  &  1  &  2  &  2  &  4  & 5  &  9  &  10 &  19 &  23  &  42  &  49 \\ \hline
$\FAMTV_w$ & 0 & 2 &  2  &  6  &  12 & 20  &  40 & 76 &   &  &  &    &    &    \\ \hline
$\AMTV_w$ & 0 & 1 &  2  &  4  &  7  &  13 &  24 & 44 &  81 &   &   &    &    &    \\ \hline
$\FAMMV_w$ & 0 & 1 &  2  &  4  &  8  &  16 &    &   &    &   &   &    &    &    \\  \hline
\end{tabular}
\end{center}

With strong numerical support, Xu and the author conjecture that $\{\dim_\Q \AMTV_w\}_{w\ge 1}$ form the
tribonacci sequence (see \cite[Conjecture~5.2]{XuZhao2020c}). For~MTVs, Kaneko and Tsumura conjecture
in \cite{KanekoTs2020} that, for~all $k\ge 1$
\begin{align*}
\dim_\Q \MTV_{2k}=\dim_\Q \MTV_{2k-1}+\dim_\Q \MTV_{2k-2}.
\end{align*}

From~numerical computation, we can formulate its finite analog as follows:
\begin{conj}\label{conj:dimFMTVs}
For all $k\ge 1$,
\begin{align*}
\dim_\Q \FMTV_{2k+1}=\dim_\Q \FMTV_{2k}+\dim_\Q \FMTV_{2k-1}.
\end{align*}
\end{conj}

\section{Conclusions}
The author's main purpose in this paper is to study the finite and symmetric MTVs and their alternating versions, which are level two and level four variations of finite MZVs. It was found that there are many nontrivial $\Q$-linear relations among these values,
such as the reversal and the linear shuffle relations. We also numerically discovered some identities, which were proposed as conjectures. Due to the limitation of computing power, we then computed the structures of MTVs (their alternating versions) when the weight was less than
7 (resp.\ 4) by using finite Euler sums and the relations that we discovered. Throughout the study, we were guided by the Conjecture Principle Philosophy \ref{cpp}, which provides the big picture in which the main objects of this paper lie. We plan to investigate the symmetric versions of these values more in a future~paper.

\bigskip
\noindent
\textbf{Acknowledgment.} This research is supported by the Jacobs Prize from The Bishop's~School. The author thanks the referees for their detailed comments which has helped improve the clarity of the paper.

\end{document}